\documentclass[smallextended]{svjour3}
\usepackage{graphicx}
\usepackage{amssymb}
\usepackage{amsmath}
\usepackage[T1]{fontenc}
\usepackage[english]{babel}
\usepackage{cite}
\DeclareMathOperator*\Ker{ker}%
\DeclareMathOperator*\Real{Re}%
\begin{document}

\title{On the best Ulam constant of the linear differential operator with constant coefficients}

\author{Alina Ramona Baias    \and
        Dorian Popa}

\institute{Alina Ramona Baias \at
              Technical University of Cluj-Napoca,
Department of Mathematics \\
G. Bari\c tiu No.25, 400027, Cluj-Napoca, Romania
           \and
           Dorian Popa \at
              Technical University of Cluj-Napoca,
Department of Mathematics \\
G. Bari\c tiu No.25, 400027, Cluj-Napoca, Romania
}

\date{Received: date / Accepted: date}

\maketitle

\begin{abstract}
The linear differential operator with constant coefficients
$$D(y)=y^{(n)}+a_1 y^{(n-1)}+\ldots+a_n y,\quad y\in \mathcal{C}^{n}(\mathbb{R}, X)$$ acting in a Banach space
$X$ is Ulam stable if and only if its characteristic equation has no roots on the imaginary axis. We prove that
if the characteristic equation of $D$ has distinct roots $r_k$ satisfying $\Real r_k>0,$ $1\leq k\le n,$ then the
best Ulam constant of $D$ is $K_D=\frac{1}{|V|}\int_{0}^{\infty}\left|\sum\limits_{k=1}^n(-1)^kV_ke^{-r_k
x}\right|dx,$ where $V=V(r_1,r_2,\ldots,r_n)$ and $V_k=V(r_1,\ldots,r_{k-1},r_{k+1}, \ldots, r_n),$ $1\leq k\leq
n,$ are Vandermonde determinants.

\end{abstract}

\section{Introduction}

In this paper, we denote by $\mathbb{K}$ the field of real numbers $\mathbb{R}$ or the field of complex numbers
$\mathbb{C}$.
 Let $M$ and $N$ be two linear spaces over the field $\mathbb{K}.$

\begin{definition}\label{def1}
A function $\rho_{M}:M\to[0,\infty]$ is called a \emph{gauge} on $M$ if the following properties hold:
\begin{enumerate}
\item[i)] $\rho_M(x)=0$ if and only if $x=0$;
\item[ii)] $\rho_M(\lambda x)=|\lambda|\rho_M(x)$ for all $x\in M,$ $\lambda\in\mathbb{K},$ $\lambda\neq0.$
\end{enumerate}
\end{definition}

Throughout this paper we denote by $(X,\|\cdot\|)$ a Banach space over the field $\mathbb{C}$ and by
$\mathcal{C}^{n}(\mathbb{R}, X)$ the linear space of all $n$ times differentiable functions with continuous $n$-th
derivatives, defined on $\mathbb{R}$ with values in $X$. $\mathcal{C}^{0}(\mathbb{R}, X)$ will be denoted as usual
by $\mathcal{C}(\mathbb{R}, X).$ For $f\in\mathcal{C}^{n}(\mathbb{R}, X)$ define
\begin{equation}\label{norma}
\|f\|_{\infty}=\sup\{\|f(t)\|:t\in \mathbb{R}\}.
\end{equation}
Then $\|f\|_{\infty}$ is a gauge on $\mathcal{C}^{n}(\mathbb{R}, X).$ We suppose that $\mathcal{C}^{n}(\mathbb{R},
X)$ and $\mathcal{C}(\mathbb{R}, X)$ are endowed with the same gauge $\|\cdot\|_{\infty}.$

Let $\rho_M,\rho_N$  be two gauges on the linear spaces $M$ and $N,$ respectively and let $L:M\to N$ be a linear
operator.

We denote by
$\Ker{L}=\{x\in M|Lx=0\}$ and $R(L)=\{Lx|x\in M\}$ the kernel and the range of the operator $L$, respectively.

\begin{definition}\label{def2}
We say that the operator $L$ is Ulam stable if there exists $K\geq0$ such that for every $\varepsilon>0$ with
$\rho_N(Lx)\leq \varepsilon$ there exists $z\in\Ker{L}$ with the property $\rho_M(x-z)\leq K\varepsilon.$
\end{definition}

The Ulam stability of the operator $L$ is equivalent with the stability of the associated equation $Lx=y,$ $y\in
R(L).$
 An element $x\in M$ satisfying $\rho_N(Lx)\leq \varepsilon$ for some positive $\varepsilon$ is called an
 approximate solution of the equation $Lx=y,$ $y\in R(L).$ Consequently, Definition \ref{def2} can be reformulated
 as follows: The operator $L$ is Ulam stable if for every approximate solution of $Lx=y,$ $y\in R(L)$ there exists
 an exact solution of the equation near it. The problem of Ulam stability can be traced back to 1940 and is due to
 Ulam \cite{ulam}. Ulam formulated this problem during a conference at Madison University, Wisconsin, for the
 equation of the homomorphisms of a metric group. The first answer to Ulam's question was given by D.H. Hyers for
 the Cauchy functional equation in Banach spaces in \cite{hyers1}. In fact, a problem of this type was formulated
 in the famous book by Polya and Szeg\"{o} for the Cauchy functional equation on the set of integers; see
 \cite{polya}. Since than, this research area received a lot of attention and was extended to the context of
 operators, functional, differential or difference equations. For a broad overview on the topic we refer the reader
 to \cite{brzdek, hyers}.

The number $K$ from Definition \ref{def2} is called an \emph{Ulam constant} of $L.$ In what follows the infimum of
all Ulam constants of $L$ is denoted by $K_L$. Generally, the infimum of all Ulam constants of the operator $L$ is
not an Ulam constant of $L$ (see \cite{hatori, popa1}) but if it is, it will be called \emph{the best Ulam
constant} of $L,$ or simply \emph{the Ulam constant} of the operator $L$. Finding the best Ulam constant of an
equation or operator is a challenging problem because it offers the best measure of the error between the
approximate and the exact solution.
For linear and bounded operators acting on normed spaces in \cite{hatori, takahasi} is given a characterization of
their Ulam stability as well as a representation of their best Ulam constant. Using this result D. Popa and I. Ra\c
sa obtained the best Ulam constant for Bernstein, Kantorovich and Stancu operators; see \cite{popa, popa2, popa3,
popa4}. For more information on Ulam stability with respect to gauges and on the best Ulam constant of linear
operators we refer the reader to \cite{brzdek, popa-rasa}.

To the best of our knowledge the first result on Ulam stability of differential equations was obtained by M.
Ob{\l}oza \cite{obloza}. Thereafter, the topic was deeply investigated by T. Miura, S. Miyajima, S.E. Takahasi in
\cite{miura, miura1, takagi} and  S. M Jung in \cite{jung}, who gave results for various differential equations and
partial differential equations.
For further details on Ulam stability we refer the reader to \cite{brzdek, hyers, ulam}.

Let $a_1,\ldots,a_n\in\mathbb{C}$ and consider the linear differential operator $D:\mathcal{C}^{n}(\mathbb{R},
X)\to \mathcal{C}(\mathbb{R}, X)$ defined by
\begin{equation}\label{operator}
D(y)=y^{(n)}+a_1 y^{(n-1)}+\ldots+a_n y,\quad y\in \mathcal{C}^{n}(\mathbb{R}, X).
\end{equation}
Denote by $P(z)=z^n+a_1 z^{n-1}+\ldots+a_n$ the characteristic polynomial of the operator $D$ and let
$r_1,\ldots,r_n$ be the complex roots of the characteristic equation $P(z)=0.$

The problem of finding the best Ulam constant was first posed by Th. Rassias in \cite{Rassias}. Since than, various
papers on this topic appeared, but there are only few results on the best Ulam constant of differential equations
and differential operators.
In the sequel we will provide a short overview of some important results concerning the Ulam stability and best
Ulam constant of the differential operator $D.$
In \cite{miura1} is proved that the operator $D$ is Ulam stable with the Ulam constant
$\frac{1}{\prod\limits_{k=1}^{n}|\Real r_k|}$ if and only if its characteristic equation has no roots on the
imaginary axis.
In \cite{popa} D. Popa and I. Ra\c sa obtained sharp estimates for the Ulam constant of the first order linear
differential operator and the higher order linear differential operator with constant coefficients. The best Ulam
constant of the first order linear differential operator with constant coefficients is obtained in \cite{miura}.
Later, A.R. Baias and D. Popa obtained the best Ulam constant for the second order linear differential operator
with constant coefficients \cite{baias-popa2}. Recent results on Ulam stability for linear differential equations
with periodic coefficients and on the best constant for Hill's differential equation were obtained by R. Fukutaka
and M. Onitsuka in \cite{fukutaka, fukutaka1}.
Important steps in finding the best Ulam constant were made also for higher order difference equations with
constant coefficients. For details we refer the reader to \cite{baias} and the references therein.

The aim of this paper is to determine the best Ulam constant for the $n$ order linear differential operator with
constant coefficients acting in Banach spaces, for the case of distinct roots of the characteristic equation.
Through this result we improve and complement some extant results in the field.

\section{Main results}

Let $a_1,\ldots,a_n\in\mathbb{C}$ and consider the linear differential operator $D:\mathcal{C}^{n}(\mathbb{R},
X)\to \mathcal{C}(\mathbb{R}, X)$ defined by
\begin{equation}\label{operator}
D(y)=y^{(n)}+a_1 y^{(n-1)}+\ldots+a_n y,\quad y\in \mathcal{C}^{n}(\mathbb{R}, X).
\end{equation}

If $r_1,r_2,\ldots,r_n $ are distinct roots of the characteristic equation $P(z)=0$, then the general solution of
the homogeneous equation $D(y)=0$ is given by

\begin{equation}\label{sol eqH}
y_H(x)=C_1e^{r_1x}+C_2e^{r_2x}+\cdots +C_ne^{r_n x},
\end{equation}
 where $\mathcal{C}_1,\ldots,\mathcal{C}_n\in X$ are arbitrary constants.
Consequently
\begin{equation}\label{sol eqH-ker}
  \Ker D=\left\{\sum\limits_{k=1}^nC_ke^{r_kx}|C_1,C_2,\ldots,C_n\in X\right\}.
\end{equation}
 The operator $D$ is surjective, so according to the variation of constants method, for every
 $f\in\mathcal{C}(\mathbb{R}, X)$ there exists a particular solution of the equation $D(y)=f$ of the form
 \begin{eqnarray*}
 y_P(x)=\sum\limits_{k=1}^nC_k(x)e^{r_kx},\quad x\in\mathbb{R},
 \end{eqnarray*}
 where $\mathcal{C}_1,\ldots, \mathcal{C}_n$ are functions of class $\mathcal{C}^{1}(\mathbb{R}, X)$ which satisfy

 \begin{equation}\label{eq3-matrice}
 \begin{pmatrix}e^{r_1x}&e^{r_2x}&\ldots& e^{r_nx}\\r_1e^{r_1x}&r_2e^{r_2x}&\ldots&r_n e^{r_nx}\\ \ldots &\ldots&
 \ldots& \ldots\\
 r_1^{n-1}e^{r_1x}&r_2^{n-1}e^{r_2x}&\ldots& r_n^{n-1}e^{r_nx}\end{pmatrix}
 \begin{pmatrix}
   C_1^{\prime}(x) \\
     C_2^{\prime}(x) \\
    \vdots\\
     C_n^{\prime}(x)
 \end{pmatrix}=
 \begin{pmatrix}
   0 \\
   \vdots \\
   0 \\
   f(x)
 \end{pmatrix}, \quad x\in\mathbb{R}.
 \end{equation}
In what follows, we denote for simplicity the Vandermonde determinants by $V:=V(r_1,r_2,\ldots,r_n)$ and
$V_k:=V(r_1,r_2,\ldots,r_{k-1},r_{k+1},\ldots, r_n),$ $1\leq k\leq n.$ Consequently, we obtain
$$C_{k}^{\prime}(x)=(-1)^{n+k}\frac{V_k}{V}e^{-r_k x}f(x),\quad k=1,\ldots,n.$$

Hence, a particular solution of the equation $D(y)=f$ is given by
\begin{equation}\label{forma sol particulara}
y_{P}(x)=\frac{1}{V}\sum\limits_{k=1}^{n}(-1)^{n+k}V_k e^{r_kx}\int_{0}^{x}f(t)e^{-r_kt}dt,\quad x\in\mathbb{R}.
\end{equation}

  The main result concerning the Ulam stability of the operator $D$ for the case of distinct roots of the
  characteristic equation is given in the next theorem.
\begin{theorem}\label{th21}
Suppose that $r_k,$ $1\leq k\leq n,$ are distinct roots of the characteristic equation with $\Real r_k\neq 0$ and
let $\varepsilon>0.$ Then for every $y\in\mathcal{C}^{n}(\mathbb{R}, X)$ satisfying
\begin{equation}\label{condyth}
\|D(y)\|_{\infty}\leq\varepsilon
\end{equation}
there exists a unique $y_H\in\Ker D$ such that
\begin{equation}\label{conddify}
\|y-y_H\|_{\infty}\leq K\varepsilon
\end{equation}
where
\begin{align}
K&=
\left\{\begin{array}{ll}
\displaystyle\frac{1}{|V|}\displaystyle\int_{0}^{\infty}\bigg|\sum\limits_{k=1}^n(-1)^kV_ke^{-r_kx}\bigg|dx,&
\mbox{ if } \Real\limits_{\substack{ 1\leq k\leq n}}r_k>0;\\
\displaystyle\frac{1}{|V|}\displaystyle\int_{0}^{\infty}\bigg|\sum\limits_{k=1}^n(-1)^kV_ke^{r_kx}\bigg|dx,& \mbox{
if } \Real\limits_{\substack{ 1\leq k\leq n}}r_k<0;\\
\displaystyle\frac{1}{|V|}\displaystyle\int_{0}^{\infty}\left(\bigg|\sum\limits_{k=1}^p(-1)^kV_ke^{-r_kx}\bigg|+\bigg|\sum\limits_{k=p+1}^n(-1)^kV_ke^{r_kx}\bigg|\right)dx,&
\mbox{ if } \begin{array}{ll}\Real\limits_{\substack{ 1\leq k\leq p}}r_k>0;\\
 \Real\limits_{\substack{ p+1\leq k\leq n}}r_k<0.
 \end{array}
\end{array} \right.\nonumber\\
&\label{const-k}
\end{align}
\end{theorem}
\begin{proof}
{\bf Existence.}
Suppose that $y\in \mathcal{C}^{n}(\mathbb{R}, X)$ satisfies (\ref{condyth}) and let $D(y)=f.$ Then
$\|f\|_{\infty}\leq\varepsilon$ and
$$y(x)=\sum\limits_{k=1}^{n}C_ke^{r_kx}+\frac{1}{V}\sum\limits_{k=1}^{n}(-1)^{n+k}V_ke^{r_kx}\int_{0}^{x}f(t)e^{-r_kt}dt,\
x\in\mathbb{R},$$
for some $C_k\in X,$ $1\leq k\leq n.$
\begin{enumerate}
\item[i)]
Let first $\Real r_k>0,$ $1\leq k\leq n$. Define $y_H\in Ker D$ by the relation
$$y_H(x)=\sum\limits_{k=1}^{n}\widetilde{C_k}e^{r_kx},x\in\mathbb{R},\quad \widetilde{C_k}\in X,$$
where
\begin{eqnarray*}
\widetilde{C_k}=C_k+(-1)^{n+k}\frac{V_k}{V}\int_{0}^{\infty}f(t)e^{-r_kt}dt,\quad 1\leq k\leq n.
\end{eqnarray*}

Since $\|f(t)e^{-r_kt}\|\leq\varepsilon|e^{-r_kt}|=\varepsilon e^{-t\Real r_k},$ $t\geq 0,$ and
$\int_{0}^{\infty}e^{-t\Real r_k}dt$ is convergent it follows that $\int_{0}^{\infty}f(t)e^{-r_kt}dt$ is
absolutely convergent, so the constants $\widetilde{C_k},$ $1\leq k\leq n$ are well defined. Then:
\begin{eqnarray*}
y(x)-y_{H}(x)&=&\sum\limits_{k=1}^{n}C_ke^{r_kx}+\frac{1}{V}\sum\limits_{k=1}^{n}(-1)^{n+k}V_ke^{r_kx}\int_{0}^{x}f(t)e^{-r_kt}dt-\sum\limits_{k=1}^{n}\widetilde{C_k}e^{r_kx}\\
&=&\sum\limits_{k=1}^{n}C_ke^{r_kx}+\frac{1}{V}\sum\limits_{k=1}^{n}(-1)^{n+k}V_k
e^{r_kx}\int_{0}^{x}f(t)e^{-r_kt}dt\\
&-&\sum\limits_{k=1}^{n}C_ke^{r_kx}-\frac{1}{V}\sum\limits_{k=1}^{n}(-1)^{n+k}V_k
e^{r_kx}\int_{0}^{\infty}f(t)e^{-r_kt}dt\\
&=&-\frac{1}{V}\sum\limits_{k=1}^{n}(-1)^{n+k}V_ke^{r_kx}\int_{x}^{\infty}f(t)e^{-r_kt}dt\\
&=&-\frac{1}{V}\sum\limits_{k=1}^{n}(-1)^{n+k}V_k\int_{x}^{\infty}f(t)e^{r_k(x-t)}dt,\quad x\in\mathbb{R}.
\end{eqnarray*}
Now, letting $t-x=u$ in the above integral we obtain
\begin{eqnarray*}
y(x)-y_H(x)&=&-\frac{1}{V}\sum\limits_{k=1}^{n}(-1)^{n+k}V_k\int_{0}^{\infty}f(u+x)e^{-r_ku}du,\\
&=& \frac{(-1)^{n+1}}{V}\int_{0}^{\infty}\left(\sum\limits_{k=1}^{n}(-1)^k V_k e^{-r_ku}\right)f(u+x)du, \quad
x\in\mathbb{R}.
\end{eqnarray*}
Hence
\begin{eqnarray*}
\|y(x)-y_{H}(x)\|&\leq &\int_{0}^{\infty}\bigg|\frac{1}{V}\sum\limits_{k=1}^{n}(-1)^{k}V_k
e^{-r_ku}\bigg|\cdot\big\|f(u+x)\big\|du,\quad x\in\mathbb{R}\\
&\leq& \frac{\varepsilon}{|V|} \int_{0}^{\infty}\bigg|\sum\limits_{k=1}^{n}(-1)^{k}V_ke^{-r_ku}\bigg|du,\quad
x\in\mathbb{R},
\end{eqnarray*}
therefore
$$\|y-y_0\|_{\infty}\leq K\varepsilon.$$
\item[ii)]

Let $\Real r_k<0,$ $1\leq k\leq n.$ The proof follows analogously, defining
$$y_H(x)=\sum\limits_{k=1}^{n}\widetilde{C_k}e^{r_kx},x\in\mathbb{R},\quad \widetilde{C_k}\in X,$$
with
\begin{eqnarray*}
\widetilde{C_k}=C_k-(-1)^{n+k}\frac{V_k}{V}\int_{-\infty}^{0}f(t)e^{-r_kt}dt\quad 1\leq k\leq n.
\end{eqnarray*}

Then
\begin{eqnarray*}
y(x)-y_{H}(x)
&=&\sum\limits_{k=1}^{n}C_ke^{r_kx}+\frac{1}{V}\sum\limits_{k=1}^{n}(-1)^{n+k}V_ke^{r_kx}\int_{0}^{x}f(t)e^{-r_kt}dt\\
&-&\sum\limits_{k=1}^{n}C_ke^{r_kx}+\frac{1}{V}\sum\limits_{k=1}^{n}(-1)^{n+k}V_ke^{r_kx}\int_{-\infty}^{0}f(t)e^{-r_kt}dt\\
&=&\frac{1}{V}\sum\limits_{k=1}^{n}(-1)^{n+k}V_k e^{r_kx}\int_{-\infty}^{x}f(t)e^{-r_kt}dt,\\
&=&\frac{1}{V}\sum\limits_{k=1}^{n}(-1)^{n+k}V_k \int_{-\infty}^{x}f(t)e^{r_k(x-t)}dt,\\
&=&\frac{(-1)^n}{V}\int_{0}^{\infty}\left(\sum\limits_{k=1}^{n}(-1)^{k}V_k e^{r_ku}\right)f(x-u)du,\quad
x\in\mathbb{R},
\end{eqnarray*}
where $u=x-t.$
Hence
\begin{eqnarray*}
\|y(x)-y_{H}(x)\|
&\leq&\frac{1}{|V|}\int_{0}^{\infty}\bigg|\sum\limits_{k=1}^{n}(-1)^{k}V_ke^{r_ku}\bigg|\cdot
\big\|f(x-u)\big\|du=K\varepsilon, x\in\mathbb{R},
\end{eqnarray*}
which entails
$$\|y-y_H\|_{\infty}\leq K\varepsilon.$$
\item[iii)]
Let $\Real r_k>0,$ $1\leq k\leq p,$ and $\Real r_k<0,$ $p+1\leq k\leq n.$ Define $y_H$ by the relation
$$y_H(x)=\sum\limits_{k=1}^{n}\widetilde{C_k}e^{r_kx},x\in\mathbb{R},\quad \widetilde{C_k}\in X,$$
with
\begin{eqnarray*}
\widetilde{C_k}&=&C_k+(-1)^{n+k}\frac{V_k}{V}\int_{0}^{\infty}f(t)e^{-r_kt}dt,\quad 1\leq k\leq p,\\
\widetilde{C_k}&=&C_k-(-1)^{n+k}\frac{V_k}{V}\int_{-\infty}^{0}f(t)e^{-r_kt}dt,\quad p+1\leq k\leq n.
\end{eqnarray*}

Then
\[y(x)-y_{H}(x)=\frac{1}{V}\sum\limits_{k=1}^{n}(-1)^{n+k}V_k
e^{r_kx}\int_{0}^{x}f(t)e^{-r_kt}dt\] \[-\frac{1}{V}\sum\limits_{k=1}^{p}(-1)^{n+k}V_k
e^{r_kx}\int_{0}^{\infty}\!f(t)e^{-r_kt}dt +\frac{1}{V}\sum\limits_{k=p+1}^{n}(-1)^{n+k}V_k e^{r_kx}\int_{-\infty}^{0}\!f(t)e^{-r_kt}dt,\]
\[=-\frac{1}{V}\!\sum\limits_{k=1}^{p}(-1)^{n+k}V_k
e^{r_kx}\!\!\int_{x}^{\infty}\!f(t)e^{-r_kt}dt+\frac{1}{V}\!\sum\limits_{k=p+1}^{n}(-1)^{n+k}V_k
e^{r_kx}\!\!\int_{-\infty}^{x}\!f(t)e^{-r_kt}dt\]
\[=-\frac{1}{V}\sum\limits_{k=1}^{p}(-1)^{n+k}V_k
\int_{x}^{\infty}f(t)e^{r_k(x-t)}dt+\frac{1}{V}\sum\limits_{k=p+1}^{n}(-1)^{n+k}V_k\int_{-\infty}^{x}\!f(t)e^{r_k(x-t)}dt.
\]
Letting $x-t=-u,$ respectively $x-t=u$ in the previous integrals it follows
\begin{align*}
y(x)-y_{H}(x)&=-\frac{1}{V}\sum\limits_{k=1}^{p}(-1)^{n+k}V_k \int_{0}^{\infty}f(x+u)e^{-r_ku}du\\
&+
\frac{1}{V}\sum\limits_{k=p+1}^{n}(-1)^{n+k}V_k \int_{0}^{\infty}f(x-u)e^{r_ku}du,\quad x\in\mathbb{R}
\end{align*}
and
\begin{align*}
\|y(x)&-y_{H}(x)\|\leq \int_{0}^{\infty}\left(\bigg|\frac{1}{V}\sum\limits_{k=1}^{p}(-1)^{n+k}V_k
e^{-r_ku}\bigg|\big\|f(x+u)\big\|\right)du\\
&+\int_{0}^{\infty}\left(
\bigg|\frac{1}{V}\sum\limits_{k=p+1}^{n}(-1)^{n+k}V_k e^{r_ku}\bigg|\big\|f(x-u)\big\|\right)du\\
&\leq\frac{\varepsilon}{|V|}\int_{0}^{\infty}\left(\bigg|\sum\limits_{k=1}^{p}(-1)^{k}V_ke^{-r_ku}\bigg|+
\bigg|\sum\limits_{k=p+1}^{n}(-1)^{k}V_k e^{r_ku}\bigg|\right)du,\quad x\in\mathbb{R}.
\end{align*}
Therefore we get $$\|y-y_0\|_{\infty}\leq K\varepsilon.$$

The existence is proved.

{\bf Uniqueness.}
Suppose that for some $y\in\mathcal{C}^{n}(\mathbb{R}, X)$ satisfying (\ref{condyth}) there exist $y_1,
y_2\in\Ker D$ such that
$$\|y-y_j\|_{\infty}\leq K\varepsilon, \mbox{ \ \ \ }j=1,2.$$
Then
$$\|y_1-y_2\|_{\infty}\leq\|y_1-y\|_{\infty}+\|y-y_2\|_{\infty}\leq 2K\varepsilon.$$
But $y_1-y_2\in\Ker D,$ hence there exist $C_k\in X,$ $1\leq k\leq n$ such that
\begin{eqnarray}
y_1(x)-y_2(x)=\sum\limits_{k=1}^nC_ke^{r_kx},\quad x\in\mathbb{R}.
\end{eqnarray}

If $(C_1,C_2,\ldots,C_n)\neq (0,0,\ldots,0),$ then
$$\|y_1-y_2\|_{\infty}=\sup\limits_{x\in\mathbb{R}}\|y_1(x)-y_2(x)\|=+\infty,$$
contradiction with the boundedness of $y_1-y_2.$ We conclude that $C_k=0,$ $1\leq k\leq n,$ therefore $y_1=y_2.$
The theorem is proved.
\end{enumerate}
\end{proof}

\begin{theorem}\label{th22}
If $r_k$ are distinct roots of the characteristic equation with $\Real r_k\neq 0,$ $1\leq k\leq n,$ then the best
Ulam constant of $D$ is given by
\begin{align}
&K_D&=
\left\{\begin{array}{ll}
\displaystyle\frac{1}{|V|}\displaystyle\int_{0}^{\infty}\bigg|\sum\limits_{k=1}^n(-1)^kV_ke^{-r_kx}\bigg|dx,&
\mbox{if} \Real\limits_{\substack{ 1\leq k\leq n}}r_k>0;\\
\displaystyle\frac{1}{|V|}\displaystyle\int_{0}^{\infty}\bigg|\sum\limits_{k=1}^n(-1)^kV_ke^{r_kx}\bigg|dx,& \mbox{if} \Real\limits_{\substack{ 1\leq k\leq n}}r_k<0;\\
\displaystyle\frac{1}{|V|}\!\displaystyle\int_{0}^{\infty}\!\left(\bigg|\!\sum\limits_{k=1}^p(-1)^kV_ke^{-r_kx}\bigg|+\bigg|\!\sum\limits_{k=p+1}^n(-1)^kV_ke^{+r_kx}\bigg|\right)dx,&
\mbox{if}\!\begin{array}{ll}\Real\limits_{\substack{ 1\leq k\leq p}}r_k>0;\\
\!\Real\limits_{\substack{ p+1\leq k\leq n}}\!r_k<0.\end{array}
\end{array} \right.\nonumber \\
&\label{const-k1}&
\end{align}
\end{theorem}
\begin{proof}

Suppose that $D$ admits an Ulam constant $K <K_D.$

\begin{enumerate}
\item[i)]
First, let $\Real r_k>0,$ $1\leq k\leq n.$ Then
$$K_{D}=\frac{1}{|V|}\int_{0}^{\infty}\left|\sum\limits_{k=1}^n(-1)^kV_ke^{-r_kx}\right|dx.$$

Let $h(x)=\sum\limits_{k=1}^n(-1)^kV_ke^{-r_kx},\,$ $x\in\mathbb{R}.$ Take $s\in X, \|s\|=1,$ and $\theta>0$
arbitrary chosen, and consider
$f:\mathbb{R}\to X$ given by
$$f(x)=\frac{\overline{h(x)}}{|h(x)|+\theta e^{-x}}s,\quad x\in\mathbb{R}.$$
Obviously, the function $f$ is continuous on $\mathbb{R}$ and $\|f(x)\|\leq 1$ for all $x\in\mathbb{R}.$

Let $\widetilde{y}$ be the solution of $D(y)=f,$ given by

\begin{equation}\label{ytilda}
\widetilde{y}(x)=\sum\limits_{k=1}^nC_ke^{r_kx}+\frac{1}{V}\sum\limits_{k=1}^n(-1)^{n+k}V_k
e^{r_kx}\int_{0}^{x}f(t)e^{-r_kt}dt
\end{equation}
with the constants
$$C_k=-(-1)^{n+k}\frac{V_k}{V}\int_{0}^{\infty}f(t)e^{-r_kt}dt,\quad 1\leq k\leq n.$$
The improper integrals in the definition of $C_k$ $1\leq k\leq n$ are obviously absolutely convergent since
$\|f(x)\|\leq 1,$ $x\in\mathbb{R},$ and $\Real r_k>0,$ $1\leq k\leq n.$
Then
\begin{eqnarray*}\label{rel17}
\widetilde{y}(x)&=&-\frac{1}{V}\sum\limits_{k=1}^n\left((-1)^{n+k}V_k\int_{0}^{\infty}f(t)e^{-r_kt}dt\right)e^{r_kx}\\
&+&
\frac{1}{V}\sum\limits_{k=1}^n(-1)^{n+k}V_k e^{r_k x}\int_{0}^{x}f(t)e^{-r_kt}dt\\
&=&-\frac{1}{V}\sum\limits_{k=1}^n(-1)^{n+k}V_k e^{r_kx}\int_{x}^{\infty}f(t)e^{-r_kt}dt\\
&=& \frac{(-1)^{n+1}}{V}\sum\limits_{k=1}^n(-1)^{k}V_k\int_{x}^{\infty}f(t)e^{r_k(x-t)}dt.
\end{eqnarray*}
Using the substitution $x-t=-u,$ $\widetilde{y}(x)$ becomes
\begin{equation}\label{ytilda2}
\widetilde{y}(x)=\frac{(-1)^{n+1}}{V}\sum\limits_{k=1}^n(-1)^{k}V_k\int_{0}^{\infty}f(x+u)e^{-r_ku}du,\quad x\in
\mathbb{R}.
\end{equation}
Since $f$ is bounded and $\Real r_k>0,$ $1\leq k\leq n,$ it follows that $\widetilde{y}(x)$ is bounded on
$\mathbb{R}$. Furthermore
 $\|D(\widetilde{y})\|_{\infty}\leq 1$ and the Ulam stability of $D$ for $\varepsilon =1$ with the constant $K,$
 leads to the existence of $y_H\in\Ker D,$ given by
$$y_{H}(x)=\sum\limits_{k=1}^nC_ke^{r_kx},\quad x\in\mathbb{R},$$ $C_k \in X,$ $1\leq k\leq n,$ with the
property
\begin{equation}\label{ytildalesthanl}
\|\widetilde{y}-y_H\|_{\infty}\leq K.
\end{equation}
If $(C_1,C_2,\ldots,C_n)\neq(0,0,\ldots,0)$ we get, in view of the boundedness of $\widetilde{y}$
\begin{equation}\label{liminf}
\lim\limits_{x\to\infty}\|\widetilde{y}(x)-y_H(x)\|=+\infty,
\end{equation}
contradiction with the existence of $K$ satisfying (\ref{ytildalesthanl}).
Therefore $C_1=C_2=\cdots=C_n=0,$ and the relation (\ref{ytildalesthanl}) becomes
\begin{eqnarray}\label{rel20}
\|\widetilde{y}(x)\|\leq K,\mbox{ for all } x\in\mathbb{R}.
\end{eqnarray}
Now let $x=0$ in (\ref{rel20}). We get, in view of (\ref{ytilda2}),

\begin{eqnarray*}
\frac{1}{|V|}\left\|\int_{0}^{\infty}\left(\sum\limits_{k=1}^n(-1)^{k}V_ke^{-r_ku}\right)f(u)du\right\|\leq K,
\end{eqnarray*}
or equivalently
\begin{equation}\label{stea}
\frac{1}{|V|}\left\|\int_{0}^{\infty}h(u)f(u)du\right\|=
\frac{1}{|V|}\int_{0}^{\infty}\frac{|h(u)|^2}{|h(u)|+\theta e^{-u}}du\leq K,\quad \forall \theta >0.
\end{equation}

Let $I(\theta)=\int_{0}^{\infty}\frac{|h(u)|^2}{|h(u)|+\theta e^{-u}}du$ and $I_0=\int_{0}^{\infty}|h(u)|du.$
We show that $\lim\limits_{\theta\to 0}I(\theta)=I_0$.
Indeed,
\begin{eqnarray*}
|I(\theta)-I_0|&\leq& \int_{0}^{\infty}\left|\frac{|h(u)|^2}{|h(u)|+\theta e^{-u}}-|h(u)|\right|du\\
&=&\theta\int_{0}^{\infty}\frac{|h(u)|e^{-u}}{|h(u)|+\theta e^{-u}}du\\
&\leq &\theta\int_{0}^{\infty}e^{-u}du=\theta,\quad \theta>0,
\end{eqnarray*}
consequently $\lim\limits_{\theta\to 0}I(\theta)=I_0.$
Letting $\theta\to 0$ in $(\ref{stea})$ we get $K_D\leq K,$ a contradiction to the supposition $K<K_D.$

\medskip

\item[ii)] The case $\Real r_k<0,$ $1\leq k\leq n,$ follows analogously.

Let $h(x)=\sum\limits_{k=1}^n(-1)^kV_ke^{r_kx},\,$ $x\in\mathbb{R}$ and $f$ be given by
$$f(x)=\frac{\overline{h(-x)}}{|h(-x)|+\theta e^{x}}s,$$ for $s\in X, \|s\|=1,$ $x\in\mathbb{R}$ and $\theta>0$
arbitrary chosen.
Obviously, the function $f$ is continuous on $\mathbb{R}$ and $\|f(x)\|\leq 1$ for all $x\in\mathbb{R}.$

Let $\widetilde{y}$ be the solution of $D(y)=f,$ given by

\begin{equation}\label{ytilda}
\widetilde{y}(x)=\sum\limits_{k=1}^nC_ke^{r_kx}+\sum\limits_{k=1}^n(-1)^{n+k}\frac{V_k}{V}e^{r_kx}\int_{0}^{x}f(t)e^{-r_kt}dt
\end{equation}
with the constants
$$C_k=(-1)^{n+k}\frac{V_k}{V}\int_{-\infty}^{0}f(t)e^{-r_kt}dt,\quad 1\leq k\leq n.$$
Using a similar reasoning as in the previous case, we obtain
$$\widetilde{y}(x)=
\frac{(-1)^{n}}{V}\int_{0}^{\infty}\left(\sum\limits_{k=1}^n(-1)^{k}V_ke^{r_ku}\right)f(x-u)du,\mbox{ }
x\in\mathbb{R}.$$
Since $f$ is bounded and $\Real r_k<0$ $1\leq k\leq n,$ it follows that $\widetilde{y}(x)$ is bounded on
$\mathbb{R}$. Furthermore
 $\|D(\widetilde{y})\|_{\infty}\leq 1$ and the Ulam stability of $D$ for $\varepsilon =1$ with the constant $K$
 leads to the existence of $y_H\in\Ker D,$ given by
$$y_{H}(x)=\sum\limits_{k=1}^n\widetilde{C_k}e^{r_kx},\quad x\in\mathbb{R},$$
$\widetilde{C_k} \in X,$ $1\leq k\leq n,$ such that
\begin{equation}\label{ytildalesthan2}
\|\widetilde{y}-y_H\|_{\infty}\leq K.
\end{equation}
If $(\widetilde{C_1},\widetilde{C_2,}\ldots\widetilde{,C_n})\neq(0,0,\ldots,0)$ it follows that
$\widetilde{y}-y_H$
is unbounded, a contradiction to the existence of $K$ satisfying (\ref{ytildalesthan2}).

Therefore $\widetilde{C_1}=\widetilde{C_2}=\cdots=\widetilde{C_n}=0,$ and the relation (\ref{ytildalesthan2})
becomes
\begin{eqnarray}\label{rel201}
\|\widetilde{y}(x)\|\leq K,\mbox{ for all } x\in\mathbb{R}.
\end{eqnarray}
Now let $x=0$ in (\ref{rel201}). We get

\begin{eqnarray*}
\frac{1}{|V|}\left\|\int_{0}^{\infty}\left(\sum\limits_{k=1}^n(-1)^{k}V_ke^{r_ku}\right)f(-u)du\right\|\leq K,
\end{eqnarray*}
or equivalently
\begin{equation}\label{stea2}
\frac{1}{|V|}\left\|\int_{0}^{\infty}h(u)f(-u)du\right\|=
\frac{1}{|V|}\int_{0}^{\infty}\frac{|h(u)|^2}{|h(u)|+\theta e^{-u}}du\leq K,\quad \forall \theta >0.
\end{equation}

Let $I(\theta)=\int_{0}^{\infty}\frac{|h(u)|^2}{|h(u)|+\theta e^{-u}}du$ and $I_0=\int_{0}^{\infty}|h(u)|du.$
The arguments used in the proof of the previous case lead to $\lim\limits_{\theta\to 0}I(\theta)=I_0$.

Letting $\theta\to 0$ in $(\ref{stea2})$ we get $K_D\leq K,$ a contradiction to the supposition $K<K_D.$

\item[iii)]

Consider $\Real r_k>0,$ $1\leq k\leq p$ and $\Real r_k<0,$ $p+1\leq k\leq n.$
Let $$h_1(x)=\sum\limits_{k=1}^p(-1)^kV_ke^{r_kx},\quad h_2(x)=\sum\limits_{k=p+1}^n(-1)^kV_ke^{-r_kx},\quad
x\in\mathbb{R}.$$

Take an arbitrary $\theta>0,$ $s\in X,$ $\|s\|=1$ and define

\begin{eqnarray}\label{f(x)3}
f(x)=
\left\{\begin{array}{ll}
\frac{\overline{h_1(-x)}}{|h_1(-x)|+\theta e^{x}}s,& \mbox{ if } x\in (-\infty,-\theta]\\
\frac{-\overline{h_2(x)}}{|h_2(x)|+\theta e^{-x}}s,& \mbox{ if } x\in [\theta,+\infty)\\
\varphi(x),& \mbox{ if } x\in (-\theta,\theta).
\end{array} \right.
\end{eqnarray}

where $\varphi:(-\theta,\theta)\to X$ is an affine function chosen such that $f$ is continuous on $\mathbb{R}.$
Remark that $\|f\|_{\infty}\leq 1.$

Let $\widetilde{y}$ be the solution of $D(y)=f,$ given by

\begin{equation}\label{ytilda}
\widetilde{y}(x)=\sum\limits_{k=1}^n C_k
e^{r_kx}+\frac{1}{V}\sum\limits_{k=1}^n(-1)^{n+k}V_ke^{r_kx}\int_{0}^{x}f(t)e^{-r_kt}dt
\end{equation}
with the constants
\begin{eqnarray*}
C_k&=&-(-1)^{n+k}\frac{V_k}{V}\int_{0}^{\infty}f(t)e^{-r_kt}dt,\quad 1\leq k\leq p.\\
C_k&=&(-1)^{n+k}\frac{V_k}{V}\int_{-\infty}^{0}f(t)e^{-r_kt}dt,\quad p+1\leq k\leq n.
\end{eqnarray*}
Consequently
\begin{equation*}
\widetilde{y}(x)=\frac{(-1)^n}{V}\displaystyle{\int}_{0}^{\infty}\bigg(\bigg(\sum\limits_{k=p+1}^{n}(-1)^kV_ke^{r_ku}\bigg)f(x-u)-\bigg(\sum\limits_{k=1}^{p}(-1)^kV_ke^{-r_ku}\bigg)f(x+u)\bigg)du.
\end{equation*}
Since $f$ is bounded, taking account of the sign of $\Real r_k,$ $1\leq k\leq n,$ it follows that
$\widetilde{y}(x)$ is bounded. The relation $\|D(y)\|_{\infty}=\|f\|_{\infty}<1$ and the stability of $D$  for
$\varepsilon =1$ with the Ulam constant $K,$ leads to the existence of an exact solution $y_H\in \Ker D$ given
by
$$y_{H}(x)=\sum\limits_{k=1}^n\widetilde{C}_ke^{r_kx},\quad x\in\mathbb{R},$$
such that
\begin{eqnarray}\label{stea3}
\|\widetilde{y}-y_H\|_{\infty}\leq K.
\end{eqnarray}

For $(\widetilde{C_1}, \widetilde{C_2},\ldots,\widetilde{C_n})\neq (0,0,\ldots,0),$ the solution $y_H$ is
unbounded, therefore the relation $(\ref{stea3})$ is true only for $y_H(x)=0$ $x\in\mathbb{R}.$
Consequently relation $(\ref{stea3})$ becomes
\begin{equation}\label{relcuyo}
\|\widetilde{y}(x)\|\leq K, \quad x\in\mathbb{R}.
\end{equation}
For $x=0$ we get $\|\widetilde{y}(0)\|\leq K.$ But
\begin{align*}
\widetilde{y}(0)&=\frac{(-1)^n}{V}\int_{0}^{\infty}\bigg(h_1(u)f(-u)-h_2(u)f(u)\bigg)du \\
&=\frac{(-1)^n}{V}\left\{\int_{\theta}^{\infty}h_1(u)f(-u)du-\int_{\theta}^{\infty}h_2(u)f(u)du\right\}\\
&+\frac{(-1)^n}{V}\int_{0}^{\theta}\bigg(h_1(u)f(-u)-h_2(u)f(u)\bigg)du\\
&=\frac{(-1)^n}{V}\left\{\int_{\theta}^{\infty}\frac{|h_1(u)|^2}{|h_1(u)|+\theta
e^{-u}}du+\int_{\theta}^{\infty}\frac{|h_2(u)|^2}{|h_2(u)|+\theta
e^{-u}}du\right\}\\
&+\frac{(-1)^n}{V}\int_{0}^{\theta}\bigg(h_1(u)f(-u)-h_2(u)f(u)\bigg)du.
\end{align*}
Analogously to the previous cases it can be proved that if $\theta\to 0$ then
\begin{eqnarray*}
\int_{\theta}^{\infty}\frac{|h_1(u)|^2}{|h_1(u)|+\theta e^{-u}}du&\longmapsto& \int_{0}^{\infty}|h_1(u)|du\\
\int_{\theta}^{\infty}\frac{|h_2(u)|^2}{|h_2(u)|+\theta e^{-u}}du&\longmapsto& \int_{0}^{\infty}|h_2(u)|du,
\end{eqnarray*}
and
$$\int_{0}^{\theta}\bigg(h_1(u)f(-u)-h_2(u)f(u)\bigg)du\longmapsto 0,$$ in view of the relation
 $$\|f(u)\|=\|\varphi(u)\|\leq 1, \quad u\in [-\theta,\theta].$$
 Hence, letting now $\theta\to 0$ in $(\ref{relcuyo})$ we get $K_D<K,$ a contradiction.
\end{enumerate}
\end{proof}

\begin{theorem}\label{thA}
If $r_k,$ $1\leq k\leq n,$ are real and distinct roots of the characteristic equation and $a_n\neq 0,$ then the
best Ulam constant of the operator $D$ is
\begin{equation}\label{ecDreal}
K_D=\frac{1}{|\prod\limits_{k=1}^{n}r_k|}=\frac{1}{|a_n|}.
\end{equation}
\end{theorem}

\begin{proof}
Suppose that $D$ admits an Ulam constant $K<K_D$.
Let $\varepsilon>0$ and $$\widetilde{y}(x)=\frac{\varepsilon}{a_n},\quad x\in \mathbb{R}.$$
Then $\|D(\widetilde{y})\|_{\infty}=\varepsilon$ and since $D$ is Ulam stable with the constant $K$ it follows that
there exists $y_H\in \Ker D$ such that
\begin{equation}\label{eq2A}
\|\widetilde{y}-y_H\|_{\infty}\leq K\varepsilon.
\end{equation}

Clearly if $y_H$ is not identically $0\in X,$ then it is unbounded so relation $(\ref{eq2A})$ cannot hold.
Therefore $y_H(x)=0$ for all $x\in \mathbb{R}$ and relation $(\ref{eq2A})$ becomes $\|\widetilde{y}\|_{\infty}\leq
K\varepsilon,$ or $K_D\leq K,$ a contradiction.
\end{proof}

The previous results lead to the following identity.

\begin{proposition}\label{thB}
If $r_k,$ $1\leq k\leq n,$ are real distinct, nonzero numbers then
\begin{align}
\frac{1}{|r_1r_2\cdots r_n|}&=K_D,
\end{align}
where $K_D$ is given by $(\ref{const-k1}).$
\end{proposition}
\begin{proof}
For real and distinct roots $r_k,$ $1\leq k\leq n,$ of the characteristic equation, the best Ulam constant is given
on one hand by
relation $(\ref{const-k1}),$ Theorem \ref{th22} and on the other hand by relation $(\ref{ecDreal})$ in Theorem
\ref{thA}.
\end{proof}

Next, we get as well an explicit representation of the best Ulam constant for the case of complex and distinct
roots of the characteristic equation having the same imaginary part.

\begin{theorem}
If the characteristic equation of $D$ admits outside of the imaginary axis distinct roots having the same imaginary
part, then the best Ulam constant of $D$ is  given by
\begin{equation}
K_D=\frac{1}{\prod\limits_{k=1}^{n}|\Real r_k|}.
\end{equation}
\end{theorem}
\begin{proof}
Suppose that $r_k=\rho_k+i\alpha,$ $\rho_k\in\mathbb{R}\setminus\{0\},$ $1\leq k\leq n,$ $\alpha\in\mathbb{R}.$
Then the best Ulam constant of $D$ is given by $(\ref{const-k1})$ with $r_k=\rho_k,$ $1\leq k\leq n,$ and
$V=V(\rho_1,\rho_2,\ldots,\rho_n),$
$V_k=(\rho_1,\ldots,\rho_{k-1},\rho_{k+1},\ldots,\rho_n).$
Now taking account of Theorem \ref{th22} it follows
$$K_D=\frac{1}{\prod\limits_{k=1}^{n}|\rho_k|}.$$
\end{proof}

Theorem \ref{th22} is an extension of the result given in \cite{baias-popa2} for
distinct roots of the characteristic equation. Indeed, the particular case $n = 2$
corresponds to the second order linear differential operator.
\begin{equation}
D(y)=y''+a_1y'+a_2y,\quad a_1,a_2\in\mathbb{C},
\end{equation}
and the best Ulam constant in this case is
\begin{eqnarray}\label{const-k2}
K_{D}=
\left\{\begin{array}{ll}
\frac{1}{|r_1-r_2|}\int_{0}^{\infty}|e^{-r_1x}-e^{-r_2x}|dx,& \mbox{ if } \Real r_1>0, \Real r_2>0,\\
\frac{1}{|r_1-r_2|}\int_{0}^{\infty}|e^{r_1x}-e^{r_2x}|dx,& \mbox{ if } \Real r_1<0, \Real r_2<0,\\
\frac{1}{|r_1-r_2|}\left|\frac{1}{\Real r_1}-\frac{1}{\Real r_2}\right|,& \mbox{ if } \Real r_1\cdot\Real r_2<0
\end{array} \right.
\end{eqnarray}

An explicit representation of $K_D$ for the second order linear differential operator with real coefficients is
given in the next theorem.

\begin{theorem}
If $D(y)=y''+a_1y'+a_2y,$ $a_1,a_2\in\mathbb{R}\setminus\{0\},$ then the best Ulam constant of the operator is
\begin{eqnarray}\label{const-k2}
K_{D}=
\left\{\begin{array}{ll}
\frac{1}{|a_2|},& \mbox{ if } a_1^2-4a_2\geq 0,\\
\frac{1}{a_2}\coth\frac{|a_1|\pi}{\sqrt{4a_2-a_1^2}},& \mbox{ if } a_1^2-4a_2<0.
\end{array} \right.
\end{eqnarray}
\end{theorem}
\begin{proof}
Let $\delta=a_1^2-4a_2.$
\begin{itemize}
\item[i)] If $\delta\geq 0$ then $r_1, r_2\in\mathbb{R}$ and in view  of \cite[Theorem 3]{baias-popa2}
    $$K_D=\frac{1}{|r_1r_2|}=\frac{1}{|a_2|}.$$
\item[ii)] If $\delta<0$ then $r_{1,2}=\alpha\pm i\beta,$ $\alpha,\beta \in\mathbb{R},$ $\beta\neq 0.$

Suppose first $\alpha>0.$ Then
\begin{eqnarray*}
K_D&=& \frac{1}{2|\beta|}\int_{0}^{\infty}e^{-\alpha x}|e^{-i\beta x}-e^{i\beta x}|dx\\
&=&\frac{1}{2|\beta|}\int_{0}^{\infty}e^{-\alpha x}|-2i \sin \beta
x|dx=\frac{1}{|\beta|}\int_{0}^{\infty}e^{-\alpha x}|\sin(|\beta|x)|dx.
\end{eqnarray*}
Now, letting $|\beta|x=t$ in the above integral we obtain, taking account of $$\int_{0}^{\infty}e^{-px}|\sin
x|dx=\frac{1}{1+p^2}\coth \frac{p\pi}{2},\quad p>0,$$
$$K_D=\frac{1}{\beta^2}\int_{0}^{\infty}e^{-\frac{\alpha}{|\beta|} t}|\sin
t|dt=\frac{1}{\alpha^2+\beta^2}\coth\frac{\alpha}{2|\beta|}\pi=\frac{1}{a_2}\coth\frac{|a_1|\pi}{\sqrt{4a_2-a_1^2}}.$$
Analogously for $\alpha<0.$
\end{itemize}
\end{proof}

It will be interesting to obtain a closed form (if possible) for the best Ulam
constant of the $n$ order differential operator also for the case of multiple roots of the characteristic equation.
This problem might be quite challenging and we intend to leave it, for the moment, as an open
problem.

\end{document}